\documentclass{amsart}
\usepackage[utf8]{inputenc}
\usepackage{amssymb}
\usepackage{amsmath}
\usepackage{color}
\newtheorem{proposition}{Proposition}

\newtheorem{definition}{Definition}
\newtheorem{theorem}{Theorem}
\newtheorem{lemma}{Lemma}
\newtheorem{remark}{Remark}

\DeclareMathOperator{\tr}{Tr}

\title[]{Nonuniform contractions and density stability results via a smooth topological equivalence}
\author[Casta\~neda]{\'Alvaro casta\~neda}
\author[Monz\'on]{Pablo Monz\'on}
\author[Robledo]{Gonzalo Robledo}
\address{Universidad de Chile, Departamento de Matem\'aticas. Casilla 653, Santiago, Chile}
\address{Universidad de la Rep\'ublica. Facultad de Ingenier\'ia, C\'odigo Postal 11300, Montevideo,  Uruguay}
\email{castaneda@uchile.cl, grobledo@uchile.cl, monzon@fing.edu.uy}
\thanks{This research has been partially supported by FONDECYT Regular 1170968}
\date{\today}

\begin{document}

\begin{abstract}
We study the smoothness and preserving orientation properties of a global and nonautonomous version of the Hartman--Grobman Theorem  when the linear system has a nonuniform contraction on the half line. The nonuniform contraction implies the existence of a density function (\emph{i.e} a dual type of Lyapunov function) for the linear system which combined with the above diffeomorphism allow us to construct a density function for the nonlinear system.
\end{abstract}

\maketitle

\section{Introduction}

The Hartman--Grobman Theorem \cite[Th. I]{Hartman-0} states that an autonomous nonlinear   system admits a local linearization around any hyperbolic equilibrium through a local homeomorphism. This result has been developed in several ways and we will focus our interest in a special case of the nonautonomous framework, where a global homeomorphism can be constructed. We are also interested in the smoothness and preserving orientation properties of the above linearization.

\subsection{Nonautonomous linearization}
In \cite{Pugh}, C. Pugh studied an example of global linearization
by considering a family of linear autonomous systems with bounded and Lipschitz nonlinear perturbations, obtaining an explicit homeomorphism.
The extension of the Pugh's global linearization result to a nonautonomous framework was carried out by K.J. Palmer \cite{Palmer73}, which firstly considered the linear system
\begin{equation}
\label{lin}
\dot{x}=A(t)x,
\end{equation}
and a family of perturbations
\begin{equation}
\label{nolin}
\dot{y}=A(t)y+f(t,y),
\end{equation}
where $A \colon \mathbb{R} \to M(n, \mathbb{R})$ is continuous and bounded while $f: \mathbb{R} \times \mathbb{R}^n \to \mathbb{R}^n$ has properties as in \cite{Pugh}. Secondly, in order to emulate the hyperbolicity condition it was assumed that (\ref{lin}) has a dichotomy property.

\begin{definition}
\label{dicotomia}
The system (\ref{lin}) has a dichotomy on the interval $J\subseteq \mathbb{R}$ if there exists an invariant projector $P(\cdot)$, a function  $D\colon [0,+\infty[\to ]0,+\infty[$ and an increasing function $\mu\colon [0,+\infty[\to [1,+\infty[$
with $\mu(0)=1$, $\mu(t)\to +\infty$ as $t\to +\infty$, such that the fundamental matrix of (\ref{lin}), namely  $X(t)$, verifies
    \begin{displaymath}
    \left\{\begin{array}{rcl}
    ||X(t,s)P(s)||\leq  D(s)\left(\frac{\mu(t)}{\mu(s)}\right)^{-\alpha} &\textnormal{if} & t\geq s, \quad t,s\in J\\\\
    ||X(t,s)[I-P(s)]||\leq  D(s)\left(\frac{\mu(t)}{\mu(s)}\right)^{\alpha} &\textnormal{if} & s\geq t, \quad t,s\in J,
    \end{array}\right.
    \end{displaymath}
with $\alpha>0$.
\end{definition}

We point out that the invariance property in the previous Definition means that
\begin{displaymath}
P(t)X(t,s)=X(t,s)P(s) \quad \textnormal{for any $t\geq s, \quad t,s\in J$}.
\end{displaymath}

In order to mimic the homeomorphism of the autonomous case, Palmer introduced
the concept of topological equivalence as follows
\begin{definition}
\label{TopEq}
The systems \textnormal{(\ref{lin})} and \textnormal{(\ref{nolin})} are topologically equivalent
on $J$ if there exists a function $H\colon J \times \mathbb{R}^{n}\to \mathbb{R}^{n}$ with the properties
\begin{itemize}
\item[(i)] If $x(t)$ is a solution of \textnormal{(\ref{lin})}, then $H[t,x(t)]$ is a solution
of \textnormal{(\ref{nolin})},
\item[(ii)] $H(t,u)-u$ is bounded in $J \times \mathbb{R}^{n}$,

\item[(iii)]  For each fixed $t\in J$, $u\mapsto H(t,u)$ is an homeomorphism of $\mathbb{R}^{n}$.
\end{itemize}
In addition, the function $G(t,u)=H^{-1}(t,u)$ has properties \textnormal{(ii)--(iii)} and maps solutions of \textnormal{(\ref{nolin})} into solutions of \textnormal{(\ref{lin})}.
\end{definition}

The seminal work of Palmer obtained sufficient conditions of topological equivalence between (\ref{lin}) and (\ref{nolin}) under the following assumptions: a) the linear system has an exponential dichotomy on $\mathbb{R}$, namely, when $D(s)=K$,
$\mu(t)=e^{t}$, $J=\mathbb{R}$ and a constant projector in the linear part, b) the nonlinear perturbation is uniformly bounded and Lipschitz continuous with respect to $x$, and c) a smallness threshold for the Lipschitz constant. Moreover, the global homeomorphism $H(t,\cdot)$ is constructed in terms of the exponential dichotomy. Later, there has been a plethora of extensions allowing less restrictive assumptions \cite{BV-2006,BV-NA,Bento, Jiang06,Reinfelds, Zhang}.

In addition, J. Shi and K. Xiong \cite{Shi} noticed that in several cases the property (iii) of the topological equivalence can be improved
obtaining that the homeomorphism $u\mapsto H(t,u)$
is uniformly continuous and/or H\"older continuous with respect to the variable $u$.

With respect to the smoothness properties of the homeomorphisms
$u\mapsto H(t,u)$ above mentioned
there are few results in comparison with the local and autonomous linearization literature. A first step in a nonautonomous framework is to consider a contraction in the linear system (\ref{lin}) such as
Hartman in \cite{Hartman60}, which has been done in \cite{CR1}, where it is assumed that the linear system has an exponential dichotomy on $\mathbb{R}$ with projector $P=I$, and a $C^{1}$ global linearization is obtained provided stronger assumptions on the nonlinearities. This result is improved in \cite{CMR} where the exponential dichotomy is only considered in $\mathbb{R}^{+}$, this restriction allows a considerably simplification on the technical assumptions of the nonlinearities, moreover it is possible to show that the homeomorphism is $C^{r}$ with $r\geq 1$ in a simpler way.

It is worth to stress that also there is a recent smoothness local result \cite{Cuong}, where contractions and expansions are involved. This approach generalizes the Sternberg theorem to the nonautonomous case and the non--resonance conditions are deduced in terms of the exponential dichotomy spectrum   (see \cite{SS,Kloeden} for details about this spectrum).

\subsection{Density functions and smooth topological equivalence}
Let us consider the nonlinear system
\begin{equation}
\label{generico}
z'=g(t,z) \quad \textnormal{with} \quad g(t,0)=0 \quad \textnormal{for any $t\in \mathbb{R}$},
\end{equation}
where $g\colon \mathbb{R}\times \mathbb{R}^{n}\to \mathbb{R}^{n}$ is such that the existence, uniqueness
and unbounded continuation of the solutions is verified.

%

In 2001, A. Rantzer \cite{Rantzer} introduced a dual concept for Lyapunov functions called density functions in an autonomous context. In order to make the article self contained, we recall the extension  stated in \cite{Schlanbusch}.

\begin{definition}
\label{density}
A density function of \textnormal{(\ref{generico})} is a function
$\rho\colon \mathbb{R}\times (\mathbb{R}^{n}\setminus \{0\})\to [0,+\infty)$ of class $C^{1}$, integrable outside any ball centered at the origin for any fixed $t\geq 0$, which satisfies
\begin{displaymath}
\frac{\partial \rho(t,z)}{\partial t}+\triangledown \cdot [\rho(t,z)g(t,z)]>0
\end{displaymath}
almost everywhere with respect to $\mathbb{R}^{n}$ and for every $t\in \mathbb{R}$, where
$$
\triangledown \cdot [\rho g] = \triangledown \rho \cdot g + \rho[\triangledown\cdot g],
$$
and $\triangledown \rho$, $\triangledown \cdot g$ denote respectively the gradient of $\rho$ and divergence
of $g$.
\end{definition}

The density functions were introduced in an autonomous framework in order to obtain sufficient conditions
for almost global asymptotic stability; we refer
to \cite{Angeli}, \cite{CG1} and \cite{Vasconcelos} for a deeper discussion
and applications. We highlight the contribution of G. Meinsma in \cite{Meinsma}, which provides a nice interpretation of this result in terms of the continuity equation from fluid mechanics.

In the above references it was proved that the existence of a density function combined with technical conditions in an autonomous context imply the almost global stability. The extension of this result to the nonautonomous case has been stated in \cite{Monzon}, \cite{Schlanbusch}:

\begin{proposition}[Theorem 4, \cite{Schlanbusch}]
Consider the system \textnormal{(\ref{generico})} such that $z=0$ is a locally stable equilibrium point.
If there exists a density function associated to \textnormal{(\ref{generico})} and is uniformly integrable in $t$ over $\{x\in \mathbb{R}^{n}\colon ||x||>\varepsilon\}$ for any $\varepsilon>0$, then for every initial time $t_{0}$, the sets
of points that are not asymptotically attracted by the origin has zero Lebesgue measure.
\end{proposition}

In order to relate topological equivalence with density functions, let us recall that the smooth global linearization results from \cite{CR1,CMR} assumed that the linear system (\ref{lin}) is contractive. Under the assumption that the linear system is uniformly asymptotically stable, P. Monz\'on \cite{Monzon} proved that there exists a density function $\rho(t,x)$ for the linear system. Later, \'A. Casta\~neda and G. Robledo \cite{CR1} showed that by combining the existence of $\rho(t,x)$ and the preserving orientation diffeomorphism, it is possible to construct a density function for any nonlinear system (\ref{nolin}) that satisfies the assumptions of the Palmer's work \cite{Palmer73}.

\subsection{Structure and novelty of the article}
The section 2 states and comment the properties of (\ref{lin}) and (\ref{nolin}) considered in this work.
The main result of the section 3 states that there exists a topological equivalence between (\ref{lin})
and (\ref{nolin}) which has a class of smoothness $C^{r}$ with $r\geq 1$. This result extends our work \cite[Theorems 1 and 3]{CMR} in terms of that the asymptotical stability of the linear system is more general than the uniform one and that considers a broader family of nonlinearities.

The section 4 shows some consequences of the existence of the above topological equivalence and is focused in some properties between the equilibria of (\ref{lin}) and (\ref{nolin}).

The section 5 states the main results of this work. Firstly, we extend the Monz\'on's \cite{Monzon} converse result by obtaining a density function for linear systems with nonuniform asymptotical stabilities. Secondly, we prove that any density function for the system (\ref{lin}) allows to construct a density function for the nonlinear system \ref{nolin} via the the preserving orientation diffeomorphism established in section 2. This result generalizes the work \cite[Theorem 3]{CR1}.

\section{Preliminaries}
In the first part of this section, we establish and comment the properties of the systems (\ref{lin}) and (\ref{nolin}), which are considered in topological equivalence literature. In particular, the linear system is contractive or asymptotically stable in a more general way than the uniform one.

In the second part, we recall  recent converse nonuniform stability result in terms of a quadratic Lyapunov type function and the contractivity conditions above stated.

\subsection{Basic assumptions}Let us recall that
$X(t)$ is any fundamental matrix of (\ref{lin}) and its corresponding transition matrix is $X(t,s)=X(t)X^{-1}(s)$. Now, we will assume that (\ref{lin}) and (\ref{nolin}) satisfy the following properties:
\begin{itemize}
      \item [{\bf{(P1)}}] The system (\ref{lin}) admits a nonuniform $(D,\mu)$--contraction, that is there exists a constant $\alpha>0$ such that
\begin{equation}
\label{Dm-con}
||X(t,s)||\leq D(s)\left(\frac{\mu(t)}{\mu(s)}\right)^{-\alpha} \quad \textnormal{with $t\geq s \geq 0$},
\end{equation}
where $D\colon [0,+\infty[\to ]0,+\infty[$ and $\mu\colon [0,+\infty[\to [1,+\infty[$ is a $C^{1}$ increasing function such that
$$
\mu(0)=1 \quad \textnormal{and} \quad \lim\limits_{t\to +\infty}\mu(t)=+\infty.
$$

    \item [{\bf{(P2)}}] There exists two functions $\beta,\gamma \colon [0,+\infty[\to ]0,+\infty[$ such that for all $t \geq 0$ and any couple $(y, \bar{y}) \in \mathbb{R}^n \times \mathbb{R}^n$ it follows that
\begin{equation}
\label{perturbacion}
 \mid f(t, y)\mid \leq \beta(t)
 \quad \textnormal{and} \quad
 \mid f(t,y) - f(t,\bar{y}) \mid \leq \gamma(t) \mid y - \bar{y}|.
\end{equation}

\item[{\bf{(P3)}}] The function $\beta\colon [0,+\infty[\to [0,+\infty[$ defined above verifies
\begin{displaymath}
\int_{0}^{t}D(s)\left(\frac{\mu(t)}{\mu(s)}\right)^{-\alpha}\beta(s)\,ds:=p <+\infty.
\end{displaymath}

\item[{\bf{(P4)}}] The function $\gamma\colon [0,+\infty[\to [0,+\infty[$ defined above verifies
\begin{displaymath}
\int_{0}^{t}D(s)\left(\frac{\mu(t)}{\mu(s)}\right)^{-\alpha}\gamma(s)\,ds:=q<1.
\end{displaymath}
\end{itemize}

We will make some comments about the above properties:
\begin{itemize}
    \item[i)] The assumption \textbf{(P1)} was previously considered in the work of \cite{Liao} and implies that (\ref{lin}) is asymptotically stable system. Nevertheless, we added a smoothness condition for $\mu(\cdot)$ for technical reasons.
    \item[ii)] The assumption \textbf{(P1)} can be seen as a particular --but distinguished-- case of dichotomy (see Definition \ref{dicotomia}) with projector $P(\cdot)=I$, which can be seen from a geometrical point of view as the stable direction of the system.
\end{itemize}
\medskip

The property of dichotomy restricted to the particular case $P(s)=I$ allows to characterize several types of asymptotic stability. It is well known (see e.g \cite[Theorem 4.11]{Khalil}) that (\ref{lin})  is uniform asymptotically stable if and only if admits a exponential dichotomy on $[t_0, +\infty).$  In the following table we show another type of asymptotic stabilities more general than uniform one which are described in terms of $(D, \mu)-$contractions.

\medskip

\begin{tabular}{|c|c|c|c|}
\hline
\bf{Asymptotic Stability}  & \bf{$D(s)$} & \bf{$\mu(t)$} & \bf{References}\\
 \hline
Uniform &  $K$ &  $e^{t}$  & \cite{Coppel}\\
\hline
Generalized Exponential &  $K$ &  $\exp\left(\int_{0}^{t}a(\tau)\,d\tau\right)$ with $a(\tau)\geq 0$  & \cite{Martin}\\
\hline
$\mu$--stability   & $K$ &  continuous & \cite{Naulin} \\
\hline
Nonuniform   & $Ke^{\varepsilon s}$ & $e^{t}$ &   \cite{BV-CMP} \\
\hline
Generalized Nonuniform  &  $K\nu(s)^{\varepsilon}$ & continuous & \cite{Bento}\\
\hline
\end{tabular}
\medskip

\begin{itemize}
    \item[iii)] The property \textbf{(P2)} has been considered by L. Jiang \cite{Jiang06} and A. Reinfelds \emph{et. al} \cite{Reinfelds} in order to construct a topological equivalence between (\ref{lin}) and (\ref{nolin}), when the linear system has dichotomies more general than the exponential one. Moreover, the particular case when $\beta(s)$ and $\gamma(s)$ are constants, was previously considered in the topological equivalence Palmer's work \cite{Palmer73}.
    \item[iv)] The properties \textbf{(P3)--(P4)} are introduced by technical reasons as in \cite{Jiang06} and \cite{Reinfelds}. In addition, it is worth to emphasize that \textbf{(P4)} implies that if (\ref{nolin}) has an equilibrium then it is unique as it was proved in \cite{CMR}.
\end{itemize}

\subsection{Stability and converse results}
The well knowns Lyapunov stability (see \cite[Example 4.21]{Khalil}) and Lyapunov's converse (see \cite[Th. 4.12]{Khalil}) theorems state that the linear system (\ref{lin}) is uniformly asymptotically stable if and only if there exists a quadratic Lyapunov function
$$
V(t,x)=x^{*}S(t)x,
$$
where $S(t)$ a $C^{1}$, bounded, positive definite symmetric matrix that satisfies the Lyapunov's equation:
\begin{displaymath}
\dot{S}(t)+A^{*}(t)S(t)+S(t)A(t)=-Q(t),
\end{displaymath}
where $Q(t)$ is continuous, symmetric and positive definite matrix.

The above results have been extended in several ways for asymptotical stabilities more general than the uniform one. In particular, we will be focused in the work of F. Liao \emph{et. al}, which proves that:
\begin{proposition}[Th.2 \cite{Liao}]
\label{2-equiv}
If there exist a $C^{1}$ positive--definite matrix $S(t)$ for any $t\geq 0$ and positive constants $C$ and $K$ such that
\begin{equation}
\label{CotaS}
||S(t)||\leq CD^{2}(t),
\end{equation}
\begin{equation}
\label{Lyap-ec}
\dot{S}(t)+A^{*}(t)S(t)+S(t)A(t)\leq -\left( I+KS(t)\right)\frac{\mu'(t)}{\mu(t)},
\end{equation}
then the system (\ref{lin}) admits a nonuniform $(D,\mu)$--contraction.

Conversely if (\ref{lin}) admits a nonuniform $(D,\mu)$--contraction and moreover there exist constants $c>0$ and $d\geq 1$ such that
\begin{equation}
\label{restriction}
||X(t,s)||\leq c \quad \textnormal{whenever $\mu(t)\leq d\mu(s)$ with $0\leq s\leq t$},  \end{equation}
then there is a $C^{1}$ positive--definite matrix satisfying (\ref{CotaS}) and (\ref{Lyap-ec}).
\end{proposition}

\section{Topological Equivalence and its Smoothness}

In order to state the following result, we recall
the property of $C^{r}$ topological equivalence on $\mathbb{R}^{+}$ introduced in \cite{CMR}.
\begin{definition}
\label{TopEqCr}
The systems \textnormal{(\ref{lin})} and \textnormal{(\ref{nolin})} are $C^{r}$ topologically equivalent on $\mathbb{R}^{+}$ if are topologically equivalent on $\mathbb{R}^{+}$ and $u\mapsto H(t,u)$ is a $C^{r}$--diffeomorphism
with $r\geq 1$ for any fixed $t\geq 0$.
\end{definition}

\begin{theorem}
\label{teorema1}
If {\bf{(P1)--\bf{(P4)}}} are satisfied
and $y\mapsto f(t,y)$ is $C^{r}$ with $r\geq 1$ for any fixed $t\geq 0$,
 then \textnormal{(\ref{lin})} and \textnormal{(\ref{nolin})} are $C^{r}$  topologically equivalent on $\mathbb{R}^{+}$.
\end{theorem}

\begin{proof}

As in \cite{CMR}, the proof we will be decomposed  in four steps to make a more readable proof . Namely, the step 1 defines two auxiliary systems whose solutions are used to construct the maps $H$ and $G$ in the step 2. To prove that these maps defines a topological equivalence, the properties (i)--(ii) are verified in the step 3, the smoothness and orientation preserving properties
are dealed in the step 4.

\noindent\emph{Step 1: Construction of the maps $H$ and $G$.} Consider the initial value problems
\begin{equation}
\label{pivote2}
\left\{\begin{array}{rcl}
w'&=&A(t)w-f(t,y(t,\tau,\eta))\\
w(0)&=& 0,
\end{array}\right.
\end{equation}
and
\begin{equation}
\label{pivote1}
\left\{\begin{array}{rcl}
z'&=&A(t)z+f(t,x(t,\tau,\xi)+z)\\
z(0)&=& 0,
\end{array}\right.
\end{equation}
where
$t \mapsto x(t,\tau,\xi)$ and $t\mapsto y(t,\tau,\eta)$ are solutions of (\ref{lin}) and (\ref{nolin}) passing through $\xi$ and $\eta$
at $t=\tau$.
It is easy to see that
\begin{equation}
\label{w-star}
w^{*}(t;(\tau,\eta))=-\int_{0}^{t}X(t,s)f(s,y(s,\tau,\eta))\,ds
\end{equation}
is the unique solution of (\ref{pivote2}) which is bounded due to \textbf{(P1)} and \textbf{(P3)}.

Let $BC(\mathbb{R}^+, \mathbb{R}^n)$ be the Banach space of bounded continuous functions with the supremum norm. Now,  for any couple $(\tau, \xi) \in \mathbb{R}^+ \times \mathbb{R}^n,$ we define the following operator $\Gamma_{(\tau, \xi)} \colon BC(\mathbb{R}^+, \mathbb{R}^n) \to BC(\mathbb{R}^+, \mathbb{R}^n) $
\begin{equation}
\phi \mapsto \Gamma_{(\tau, \xi)} \phi := \displaystyle \int_0^t X(t,s) f(s,x(s,\tau,\xi) + \phi) \, ds,
\end{equation}
which is well defined by \textbf{(P3)}.

By {\bf{(P2)}} and \textbf{(P4)}, we see that the  operator $\Gamma_{(\tau, \xi)} $ is a contraction. This fact combined with the Banach contraction principle implies that
\begin{displaymath}
z^{*}(t;(\tau,\xi))=\int_{0}^{t}X(t,s)f(s,x(s,\tau,\xi)+z^{*}(s;(\tau,\xi))) \, ds
\end{displaymath}
is the unique solution of (\ref{pivote1}).

Moreover, the uniqueness of solutions allows to prove that
\begin{equation}
\label{identity1}
z^{*}(t;(\tau,\xi))=z^{*}(t;(r,x(r,\tau,\xi))) \quad \textnormal{for any $r\geq 0$},
\end{equation}
and
\begin{equation}
\label{identity2}
w^{*}(t;(\tau,\nu))=w^{*}(t;(r,y(r,\tau,\nu))) \quad \textnormal{for any $r\geq 0$}.
\end{equation}

Now, for any $t\geq 0$ we define the maps $H(t,\cdot)\colon \mathbb{R}^{n}\to \mathbb{R}^{n}$ and
$G(t,\cdot)\colon \mathbb{R}^{n}\to \mathbb{R}^{n}$ as follows:
\begin{displaymath}
\begin{array}{rcl}
H(t,\xi)&:=&\displaystyle  \xi+\int_{0}^{t}X(t,s)f(s,x(s,t,\xi)+z^{*}(s;(t,\xi))\,ds \\\\
&=&\xi + z^{*}(t;(t,\xi)),
\end{array}
\end{displaymath}
and
\begin{equation}
\label{Homeo-G}
\begin{array}{rcl}
G(t,\eta)&:=&\displaystyle \eta -\int_{0}^{t}X(t,s)f(s,y(s,t,\eta))\,ds \\\\
&=&\eta+w^{*}(t;(t,\eta)).
\end{array}
\end{equation}
By using (\ref{identity1}), we can verify that
\begin{displaymath}
\begin{array}{rcl}
H[t,x(t,\tau,\xi)]&=&\displaystyle  x(t,\tau,\xi)+\int_{0}^{t}X(t,s)f(s,x(s,t,x(t,\tau,\xi)+z^{*}(s;(t,x(t,\tau,\xi))))\,ds \\\\
&=& \displaystyle x(t,\tau,\xi)+\int_{0}^{t}X(t,s)f(s,x(s,\tau,\xi)+z^{*}(s;(\tau,\xi)))\,ds \\\\
&=&x(t,\tau,\xi)+z^{*}(t;(\tau,\xi)).
\end{array}
\end{displaymath}

\noindent\emph{Step 2: $H$ and $G$ satisfy properties (i)--(ii) of Definition \ref{TopEq}.}
By (\ref{lin}) and (\ref{pivote1}) combined with the above equality, we have that
$$
\begin{array}{rcl}
\displaystyle \frac{\partial }{\partial t} H[t,x(t,\tau,\xi)] & = & \displaystyle \frac{\partial }{\partial t} x(t, \tau, \xi) + \frac{\partial }{\partial t} z^*(t; (\tau, \xi))\\\\
& = & A(t)x(t,\tau,\xi) + A(t)z^*(t; (\tau, \xi)) + f(t,H[t,x(t,\tau,\xi)])\\\\
& = & A(t)H[t,x(t,\tau,\xi)] + f(t,H[t,x(t,\tau,\xi)]),
\end{array}
$$
then $t\mapsto H[t,x(t,\tau,\xi)]$ is solution of (\ref{nolin}) passing through
$H(\tau,\xi)$ at $t=\tau$. As consequence of uniqueness of solution we obtain
\begin{equation}\label{conj1}
H[t, x(t,\tau, \xi)] = y(t,\tau, H(\tau,\xi)),
\end{equation}
similarly, it can be proved that $t\mapsto G[t,y(t,\tau,\eta)]$ is solution of (\ref{lin}) passing through $G(\tau, \eta)$
at $t=\tau$ and

\begin{equation}\label{conj2}
G[t, y(t,\tau, \eta)] = x(t,\tau, G(\tau,\eta)) = X(t, \tau)G(\tau,\eta),
\end{equation}
and the property (i) follows.
Secondly, by using \textbf{(P2)--(P3)} it follows that
\begin{displaymath}
|H(t,\xi)-\xi|\leq  \int_{0}^{t}D(s)\left(\frac{\mu(t)}{\mu(s)}\right)^{-\alpha}\beta(s)\,ds <+\infty,
\end{displaymath}
for any $t\geq 0$. A similar inequality can be obtained for
$|G(t,\eta)-\eta|$ and the property (ii) is verified.

\noindent\emph{Step 3: $G$ is bijective for any $t\geq 0$.} We will
first show that $H(t, G(t, \eta)) = \eta$ for any $t\geq 0$. Indeed,
$$
\begin{array}{rcl}
H[t,G[t,y(t,\tau,\eta)]] & = & G[t,y(t,\tau,\eta)] \\\\ & &  + \displaystyle \int_0^t X(t,s) f(s,x(s,t,G[t,y(t,\tau,\eta)])+z^*(s;(t,G[t,y(t,\tau,\eta)]))) \, ds\\\\
& = &  y(t,\tau, \eta)  -\displaystyle \int_0^t  X(t,s)f(s,y(s,\tau, \eta)) \, ds\\\\
&&+\displaystyle \int_0^t X(t,s) f(s,x(s,t,G[t,y(t,\tau,\eta)])+z^*(s;(t,G[t,y(t,\tau,\eta)]))) \, ds.
\end{array}
$$
Let $\omega(t) = | H[t, G[t,y(t,\tau,\eta)]] - y(t,\tau, \eta)| .$ Hence by using {\bf{(P1)}} and {\bf{(P2)}} we have that
$$
\begin{array}{ll}
 \omega(t) & =  \left |\displaystyle \int_0^t X(t,s) \{f(s,x(s,t,G[t,y(t,\tau,\eta)])+z^*(s;(t,G[t,y(t,\tau,\eta)]))) -f(s,y(s,\tau, \eta))\} \, ds \right |\\\\
&\leq  \displaystyle \int_0^t D(s)\left(\frac{\mu(t)}{\mu(s)}\right)^{-\alpha} \hspace{-0.5cm}\gamma(s)|\{x(s,t,G[t,y(t,\tau,\eta)])+z^*(s;(t,G[t,y(t,\tau,\eta)])) -y(s,\tau, \eta)\}| \, ds.
\end{array}
$$
Notice that,
$$x(s,t,G[t,y(t,\tau,\eta)])+z^*(s;(t,G[t,y(t,\tau,\eta)])) = H[s,x(s,t,G[t,y(t,\tau,\eta)])]$$
and recalling that
$$x(s,t,G[t,y(t,\tau,\eta)]) = x(s, \tau, G(\tau, \eta)) = G[s,y(s,\tau, \eta)],$$
we can see
$$ H[s,x(s,t,G[t,y(t,\tau,\eta)])] = H[s,G[s,y(s,\tau,\eta)]].$$
Therefore, by \textbf{(P4)} we obtain
$$\omega(t) \leq \int_0^t D(s)\left(\frac{\mu(t)}{\mu(s)}\right)^{-\alpha}\hspace{-0.3cm}\gamma(s) \omega(s) \, ds \leq  \displaystyle q\sup_{s \in \mathbb{R}^+} \{\omega(s)\} \quad \rm{for \,\, all} \quad t \geq 0.$$

The supremum is well defined by property (i) and the fact that all the solutions of systems \eqref{lin} and \eqref{nolin} are bounded on $\mathbb{R}^+$. Now, we take the supremum on the left side above and it follows that $\omega(t) = 0$ for any $t \geq 0$ since $0<q<1$. In particular, when we take $t = \tau$ we obtain  $H(\tau, G(\tau, \eta)) = \eta.$

Next, we will prove that $G(t, H(t, \xi)) = \xi.$ In fact, due to (\ref{conj1}) we have that
$$
\begin{array}{rcl}
G[t,H[t,x(t,\tau,\xi)]] & = & H[t,x(t,\tau,\xi)] \\\\ & &  - \displaystyle \int_0^t X(t,s) f(s,y(s,t,H[t,y(x,\tau,\xi)])) \, ds\\\\
&= &  x(t,\tau, \xi) + \\\\
&& \displaystyle \int_0^t  X(t,s) \{f(s,H[s,x(s,\tau,\xi)]) - f(s,y(s,\tau,H(\tau,\xi)))\} \, ds\\\\
& = &  x(t,\tau, \xi).
\end{array}
$$
and taking $t=\tau$ leads to $G(\tau,H(\tau,\xi))=\xi$. In consequence, for any $t\geq 0$, $H$ is a bijection and $G$ is its inverse.

\medskip

\noindent\emph{Step 4: $G$ is a preserving orientation diffeomorphism for any $t\geq 0$.} As  $y\mapsto f(t,y)$ is $C^{r}$ with $r\geq 1$ for any $t\geq 0$, we can use classical results (see \emph{e.g.}, Theorem 4.1 from \cite[Ch.V]{Hartman}) to see that
the map
$\eta \mapsto y(t,\tau,\eta)$ is also $C^{r}$ for any  fixed couple $(t,\tau)$. Then, as $y\mapsto f(t,y)$ is $C^{1}$, it follows that $y\mapsto Df(t,y)$ and  $\eta \mapsto \partial y/\partial \eta$ are continuous. Indeed, we can see that this last map satisfy the matrix differential equation
\begin{equation}
\label{MDE1}
\left\{
\begin{array}{rcl}
\displaystyle \frac{d}{dt}\frac{\partial y}{\partial\eta}(t,\tau,\eta)&=&\displaystyle \{A(t)+Df(t,y(t,\tau,\eta))\}\frac{\partial y}{\partial \eta}(t,\tau,\eta),\\\\
\displaystyle \frac{\partial  y}{\partial\eta}(\tau,\tau,\eta)&=&I.
\end{array}\right.
\end{equation}

Now, it is easy to see that
\begin{equation}
\label{derivada-parcial}
\frac{\partial G}{\partial \eta_{i}}(t,\eta)=e_{i}-\int_{0}^{t}X(t,s)Df(s,y(s,t,\eta))\frac{\partial y}{\partial \eta_{i}}(s,t,\eta)\,ds \quad (i=1,\ldots,n),
\end{equation}
which implies that the partial derivatives exists and are continuous for any fixed $t\geq 0$, then $\eta \mapsto G(t,\eta)$ is $C^{1}$.

By using the identity $X(t,s)A(s)=-\frac{\partial }{\partial s}X(t,s)$ combined with
(\ref{MDE1}) we can deduce that for any $t\geq 0$, the Jacobian matrix of $\eta \mapsto G(t,\eta)$ is given by
\begin{equation}
    \label{machine}
\begin{array}{rcl}
\displaystyle\frac{\partial G}{\partial\eta}(t,\eta)&=&  \displaystyle I-\int_{0}^{t}X(t,s)Df(s,y(s,t,\eta))
\frac{\partial y}{\partial\eta}(s,t,\eta)\,ds \\\\
&=&I - \displaystyle \int_{0}^{t}\frac{d}{ds}\left\{X(t,s)\frac{\partial y}{\partial \eta}(s,t,\eta)\right\}\,ds \\\\
&=&\displaystyle X(t,0)\frac{\partial y(0,t,\eta)}{\partial \eta},
\end{array}
\end{equation}
and Theorems 7.2 and 7.3 from  \cite[Ch.1]{Coddington}  imply that
$Det \frac{\partial G(t,\eta)}{\partial \eta}>0$ for any $t\geq 0$.

Summarizing, we have that $\eta \mapsto G(t,\eta)$ is $C^{1}$ and its Jacobian matrix has a non vanishing determinant. In addition, let us recall that
$$
G(t,\eta)=\eta + w^{*}(t;(t,\eta)),
$$
where $w^{*}(t;(t,\eta))$ is given by (\ref{w-star}). Since \textbf{(P3)} implies  $|w^{*}(t;(t,\eta))|\leq p$ for any $(t,\eta)$, we can see that
$|G(t,\eta)|\to +\infty$ as $|\eta|\to +\infty$. Therefore, by Hadamard's Theorem (see \emph{e.g} \cite{Plastock, Radulescu}), we conclude that $\eta \mapsto G(t,\eta)$ is a global preserving orientation diffeomorphism for any fixed $t\geq 0$.

Finally, as $\eta \to y(0,t,\eta)$ is $C^{r}$, we use the identity (\ref{machine}) to deduce that the $m$--th partial derivatives of $\eta \mapsto G(t,\eta)$ for any fixed $t\geq 0$ are given by
\begin{displaymath}
\frac{\partial^{|m|} G(t,\eta)}{\partial\eta_{1}^{m_{1}}\cdots \partial \eta_{n}^{m_{n}}}=X(t,0)\frac{\partial^{|m|} y(0,t,\eta)}{\partial\eta_{1}^{m_{1}}\cdots \partial \eta_{n}^{m_{n}}}, \quad \textnormal{where $|m|=m_{1}+\ldots+m_{n}\leq r$},
\end{displaymath}
and the Theorem follows.
\end{proof}

\begin{remark}
We point out the formal similarity with the proofs
of Theorems 1 and 3 from \cite{CMR}. However, there exists subtle technicalities induced by the nonuniform contraction and the properties of the nonlinearities.
\end{remark}

\section{Some consequences of the topological equivalence}

This section is devoted to study the relation between the equilibria of (\ref{lin}) and (\ref{nolin}) when are topologically equivalent on the positive half line. It is important to emphasize that there exist nonlinear systems (\ref{nolin}) satisfying \textbf{(P1)--(P4)} which does not have the same equilibrium for all $t\geq 0$ as shown by Jiang in \cite[p.487]{Jiang06}.

\begin{lemma}
\label{unicidad}
Assume that \textnormal{\textbf{(P1)--(P4)}} are fulfilled. If (\ref{nolin}) has an equilibrium then it is unique.
\end{lemma}
\begin{proof}
Let us assume that (\ref{nolin}) has two different equilibria $\bar{y}_{1}$ and $\bar{y}_{2}$, then it follows that
\begin{displaymath}
\bar{y}_{i}=X(t,0)\bar{y}_{i}+\int_{0}^{t}X(t,s)f(s,\bar{y}_{i})\,ds \quad \textnormal{for any $t\geq 0$ and $i=1,2$}.
\end{displaymath}

Now, by \textbf{(P1),(P2)} and \textbf{(P4)} we can deduce that
\begin{displaymath}
\begin{array}{rcl}
\displaystyle |\bar{y}_{1}-\bar{y}_{2}| &\leq& \displaystyle D(0)\left(\frac{\mu(t)}{\mu(0)}\right)^{-\alpha}|\bar{y}_{1}-\bar{y}_{2}|+\int_{0}^{t}D(s) \left(\frac{\mu(t)}{\mu(s)}\right)^{-\alpha}\gamma(s)|\bar{y}_{1}-\bar{y}_{2}|\,ds \\\\
& \leq & \displaystyle \left(D(0)\left(\frac{\mu(t)}{\mu(0)}\right)^{-\alpha}+q\right)|\bar{y}_{1}-\bar{y}_{2}|,
\end{array}
\end{displaymath}
which implies that $1\leq  q$ for bigger values of $t$, obtaining a contradiction.
\end{proof}

\begin{lemma}
\label{HPF}
Assume that \textnormal{\textbf{(P1)--(P4)}} are fulfilled.
\begin{itemize}
\item[(i)] If $\bar{y}=0$ is equilibrium of \textnormal{(\ref{nolin})}, namely $f(t,0)=0$ for any $t\geq 0$, then
$$
H(t,0)=G(t,0)=0 \quad \textnormal{for any} \quad t\geq 0.
$$
\item[(ii)]
If the system \textnormal{(\ref{nolin})} has a equilibrium  $\bar{y}\neq 0$, then
\begin{displaymath}
\lim\limits_{t\to +\infty}G(t,\bar{y})=0.
\end{displaymath}
\item[(iii)]
If the system \textnormal{(\ref{nolin})} has a equilibrium  $\bar{y}\neq 0$ and
\begin{equation}
\label{dominacion}
\displaystyle \lim\limits_{t\to+\infty}\mu^{-\alpha}(t)\exp\left(\int_{0}^{t}D(s)\gamma(s)\,ds\right)=0,
\end{equation}
then
\begin{displaymath}
\lim\limits_{t\to +\infty}H(t,0)=\bar{y}.
\end{displaymath}
\end{itemize}
\end{lemma}

\begin{proof}
If $f(t,0)=0$ for any $t\geq 0$ then (\ref{pivote2}) becomes (\ref{lin}) and $w^{*}(t;(\tau,0))=0$ for any $t\geq 0$ and by the definition of $G(t,\cdot)$, we have that $G(t,0)=0$ and (i) follows.

Now, let us assume that $\bar{y}\neq 0$ is the unique equilibrium of (\ref{nolin}). Then, the initial value problem (\ref{pivote2}) becomes
\begin{displaymath}
\left\{\begin{array}{rcl}
w'&=&A(t)w-f(t,\bar{y})\\
w(0)&=& 0,
\end{array}\right.
\end{displaymath}
whose solution is given by
\begin{displaymath}
\begin{array}{rcl}
w^{*}(t;(\tau,\bar{y})) &=& \displaystyle         -\int_{0}^{t}X(t,s)f(s,\bar{y})\,ds \\\\
 &=&\displaystyle \int_{0}^{t}X(t,s)A(s)\bar{y}\,ds\\\\
 &=&\displaystyle  -\int_{0}^{t}\frac{\partial}{\partial s}X(t,s)\bar{y}\,ds \\\\
 &=& (X(t,0)-I)\bar{y},
\end{array}
\end{displaymath}
and by using definition of $G(t,\cdot)$ we obtain $G(t,\bar{y})=X(t,0)\bar{y}$ and it follows by \textbf{(P1)} that $\lim\limits_{t\to +\infty}G(t,\bar{y})=0$.

Similarly, if $\xi=0$, the initial value problem (\ref{pivote1}) becomes
\begin{displaymath}
\left\{\begin{array}{rcl}
z'&=&A(t)z+f(t,z)\\
z(0)&=& 0,
\end{array}\right.
\end{displaymath}
which is not parameter dependent and its solution is
\begin{displaymath}
z^{*}(t)=\int_{0}^{t}X(t,s)f(s,z^{*}(s))\,ds.
\end{displaymath}

By the definition, we know that $H(t,0)=z^{*}(t)$, and as $\bar{y}$ is a fixed point, we have that
\begin{displaymath}
H(t,0)-\bar{y}  =  \displaystyle  -X(t,0)\bar{y}+\int_{0}^{t}X(t,s)\{f(s,H(s,0))-f(s,\bar{y})\}\,ds,
\end{displaymath}
which implies that
\begin{displaymath}
\begin{array}{rcl}
|H(t,0)-\bar{y}| & \leq  & \displaystyle  D(0)\left(\frac{\mu(t)}{\mu(0)}\right)^{-\alpha}|\bar{y}|+\int_{0}^{t}
D(s)\left(\frac{\mu(t)}{\mu(s)}\right)^{-\alpha}\gamma(s) |H(s,0)-\bar{y}|\,ds.
\end{array}
\end{displaymath}

By Gronwall's inequality we can deduce that
\begin{displaymath}
|H(t,0)-\bar{y}|\leq D(0)\mu^{-\alpha}(t)\exp\left(\int_{0}^{t}D(s)\gamma(s)\,ds\right)|\bar{y}|.
\end{displaymath}

Finally, by using (\ref{dominacion}) we obtain $\lim\limits_{t\to +\infty}H(t,0)=\bar{y}$ and the result follows.
\end{proof}

\begin{remark}
The condition (\ref{restriction}) follows immediately when the linear system has an exponential dichotomy with the identity as projector, \emph{i.e}, uniform asymptotical stability. Indeed, when we consider $D(s)=K$ and $\mu(t)=e^{t}$, the property \textbf{(P4)} becomes $K\gamma/\alpha<1$ and (\ref{restriction}) is a direct consequence from the Gronwall's Lemma.
\end{remark}

\begin{remark}
In Theorem 1 from \cite{CMR} we prove that if the linear system (\ref{lin}) is uniformly asymptotically stable, then the homeomorphism $H(t, \cdot)$ is uniformly continuous. In addition, when the nonlinear system has a unique equilibrium, the uniform continuity of $H(t, \cdot)$ combined with the property (\ref{restriction})  imply that the equilibrium of (\ref{nolin}) is also uniform asymptotically stable (see \cite[Theorem 2]{CMR}). As in a general case, the condition (\ref{restriction})  is not always verified, the preservation of asymptotic stability via homeomorphism $H(t, \cdot)$ can not be established.
\end{remark}

\section{Stability results in terms of density functions}

From now on, we will assume that $f(t,0)=0$ for any $t\geq 0$. By Lemma \ref{unicidad}, it follows that the origin is the unique equilibrium of (\ref{nolin}).

\subsection{A density converse result for linear systems}

The following result generalizes a density converse stability result obtained by P. Monz\'on in \cite[Prop.2.3]{Monzon},
which assumed that (\ref{lin}) is uniformly asymptotically stable.

\begin{theorem}
\label{densidad-lineal}
Assume that the system (\ref{lin})
satisfies the property \textnormal{\textbf{(P1)}} and its transition matrix satisfies (\ref{restriction}), then there exists a density function associated to (\ref{lin}).
\end{theorem}

\begin{proof}
By Proposition \ref{2-equiv}, there exist a $C^{1}$ positive--definite matrix $S(t)$ satisfying (\ref{CotaS}) and (\ref{Lyap-ec}). Now let us define
\begin{equation}
\label{densidad-lineal2}
\rho(t,x)=V(t,x)^{-a} \quad \textnormal{with $V(t,x)=x^{*}S(t)x$ and $a>0$}.
\end{equation}

By using the identity
\begin{displaymath}
\triangledown V(t,x)\cdot A(t)x=x^{*}\left[A^{*}(t)S(t)+S(t)A(t)\right]x,
\end{displaymath}
we can deduce that
\begin{displaymath}
\begin{array}{rcl}
\displaystyle \frac{\partial \rho}{\partial t}+\triangledown \cdot (\rho(t,x) A(t)) &=&
\displaystyle \frac{\partial \rho(t,x)}{\partial t}+
\triangledown \rho(t,x)\cdot A(t)x+\rho(t,x)\tr A(t) \\\\
&=&
\displaystyle
-a V(t,x)^{-(a+1)}\left(\frac{\partial V}{\partial t}(t,x)+\triangledown V(t,x)A(t)x\right) \\\\
\displaystyle
&&+V(t,x)^{-a}\tr A(t) \\\\
&=&V(t,x)^{-(a+1)}L(S(t),A(t),x),
\end{array}
\end{displaymath}
where
\begin{displaymath}
L(S(t),A(t),x)=x^{*}\left(-a\left\{\dot{S}(t)+A^{*}(t)S(t)+S(t)A(t)\right\}+\tr A(t)S(t)\right)x.
\end{displaymath}

By using (\ref{Lyap-ec}), combined with boundedness of $A(t)$ and positive--definiteness of $S(t)$, we choose $a>0$ big enough such that
\begin{displaymath}
\begin{array}{rcl}
L(S(t),A(t),x) &\geq &
\displaystyle
x^{*}\left\{a \left( I+KS(t)\right)\frac{\mu'(t)}{\mu(t)}+\tr A(t)S(t)\right\}x \\\\
\displaystyle
& = & x^{*}\left\{a  \frac{\mu'(t)}{\mu(t)}I+S(t)\left[a K \frac{\mu'(t)}{\mu(t)}+\tr A(t)\right]\right\}x \\\\
&>&0,
\end{array}
\end{displaymath}
therefore the function $\rho(t,x)$ is positive. Finally, the integrability of $x\mapsto \rho(t,x)$ in any domain $||x||\geq r>0$ is also verified by choosing some constant $a>0$ big enough and the result follows.
\end{proof}

\begin{remark}
As a consequence of (\ref{CotaS}) we can see that the density function of (\ref{lin}) satisfies the following inequality
\begin{displaymath}
\left(C D^{2}(t)\right)^{-\alpha}||x||^{-2\alpha}\leq \rho(t,x).
\end{displaymath}
\end{remark}

\subsection{Density function for the nonlinear system via the Diffeomorphism $G$}

The main result states that if (\ref{lin}) and (\ref{nolin}) are $C^{2}$ topologically equivalents on $\mathbb{R}^{+}$ then we can construct a density function for the nonlinear system in terms of the density function for (\ref{lin}) and the diffeomorphism $G(t,\cdot)$.

\begin{theorem}
\label{ex-dens0}
If \textnormal{\textbf{(P1)--(P4)}} and (\ref{restriction}) are satisfied, and $f(t,\cdot)$ is of class $C^{2}$ for any $t\geq 0$,
then there exists a density function $\bar{\rho}\in
C(\mathbb{R}^{+}\times (\mathbb{R}^{n}\setminus\{0\}),[0,+\infty))$
associated to \textnormal{(\ref{nolin})}, defined by
\begin{equation}
\label{dens-nl}
\bar{\rho}(t,\eta)=\rho(t,G(t,\eta))\det\frac{\partial G(t,\eta)}{\partial \eta},
\end{equation}
where $G(t,\cdot)$ is the $C^{2}$ preserving orientation diffeomorphism previously defined.
\end{theorem}

\begin{proof}
From now on we will denote $|A|$ instead of $\det A$ by visualization reasons. We shall prove that the function $\bar{\rho}(t,x)$ satisfies the properties of Definition \ref{density}
with $g(t,x)=A(t)x+f(t,x)$. Indeed, $\bar{\rho}$ is nonnegative since $\rho$ is nonnegative
and $G$ is preserving orientation. In addition, $\bar{\rho}(t,\cdot)$ is $C^{1}$ since $G(t,\cdot)$ is $C^{2}$ for any fixed $t\geq 0$.

The rest of the proof will be decomposed in several steps:

\noindent\emph{Step 1:} $\bar{\rho}(t,x)$ is integrable outside any ball centered in the origin.

By Theorem \ref{teorema1} we know that the maps $H(t,\cdot)$ and $G(t,\cdot)$ satisfies the properties of Definition \ref{TopEq}. As $f(t,0)=0$, the  statement (i) from Lemma \ref{HPF} implies $G(t,0)=0$, this fact combined with statements (ii) and (iii) from Definition \ref{TopEq} allows to
conclude that if $B$ is an open ball centered at the origin then $G(t,B)$ is an open and bounded set containing the origin. In consequence, for any fixed $t\geq 0$, the outside of $B$ is mapped in the
outside of another ball centered at the origin and contained in $G(t,B)$.

Let $\mathcal{Z}$ be a measurable set whose closure does not contain the origin. The property stated above implies that
$G(t,\mathcal{Z})$ is outside of some ball centered at the origin. Now, by the change of variables
theorem, we can see that
$$
\displaystyle \int_{\mathcal{Z}}\bar{\rho}(t,\eta)\,d\eta
=
\int_{\mathcal{Z}}\rho(t,G(t,\eta))\Big|\frac{\partial G(t,\eta)}{\partial \eta}\Big|\,d\eta
=
\int_{G(t,\mathcal{Z})}\rho(t,y)\,dy.
$$

Finally, as $\rho(t,\cdot)$ is integrable outside any open ball centered at the origin, the same follows for $\bar{\rho}(t,\cdot)$.

\noindent\emph{Step 2:} $\bar{\rho}(t,\eta)$ verifies the following property
\begin{equation}
\label{postividad}
\frac{\partial \bar{\rho}}{\partial t}(t,\eta)+\triangledown\cdot (\bar{\rho}g)(t,\eta)>0 \quad
\textnormal{a.e. in $\mathbb{R}^{n}$}.
\end{equation}

Firstly, let $\sigma =\tau +t$ and recall that $\sigma\mapsto y(\sigma,t,\eta)$ is the solution of (\ref{nolin}) passing trough $\eta$ at time $\sigma=t$. By the Liouville's formula (see \emph{e.g.}, \cite[Corollary 3.1]{Hartman}), we know that
\begin{displaymath}
\frac{\partial}{\partial \sigma}\Big|\frac{\partial y(\tau+t,t,\eta)}{\partial \eta}\Big|\Bigg|_{\tau=0}=\triangledown \cdot g(t,\eta).
\end{displaymath}

Now, it is easy to verify that:
\begin{displaymath}
\begin{array}{rcl}
\displaystyle \frac{\partial \bar{\rho}}{\partial t}(t,\eta)+\triangledown \cdot (\bar{\rho}g)(t,\eta)&=&
\displaystyle \frac{\partial}{\partial \sigma}\Big\{\bar{\rho}(\tau+t,y(\tau+t,t,\eta))\displaystyle \Big|\frac{\partial y(\tau+t,t,\eta)}{\partial  \eta}\Big|\Big\}\Big|_{\tau=0}\\\\
&=&\displaystyle \frac{\partial}{\partial \sigma}\Big\{\rho(\tau+t,G[\tau+t,y(\tau+t,t,\eta)])\\\\\
& & \displaystyle \Big|\frac{\partial G[\tau+t,y(\tau+t,t,\eta)]}{\partial y(\tau+t,t,\eta)}\Big|\Big|\frac{\partial y(\tau+t,t,\eta)}{\partial \eta}\Big|\Big\}\Big|_{\tau=0}\\\\
&=& \displaystyle \frac{\partial}{\partial \sigma}\Big\{\rho(\tau+t,G[\tau+t,y(\tau+t,t,\eta)])\\\\
& & \displaystyle \Big|\frac{\partial G[\tau+t,y(\tau+t,t,\eta)]}{\partial \eta}\Big|\Big\}\Big|_{\tau=0}.\\\\
\end{array}
\end{displaymath}

Secondly, a consequence of (\ref{conj2}) is
\begin{displaymath}
G[\tau+t,y(\tau+t,t,\eta)]=X(\tau+t,t)G(t,\eta),
\end{displaymath}
which implies:
\begin{displaymath}
\begin{array}{rcl}
\displaystyle \frac{\partial \bar{\rho}}{\partial t}(t,\eta)+\triangledown \cdot (\bar{\rho}g)(t,\eta) &=& \displaystyle  \frac{\partial}{\partial \sigma}\Big\{\rho(\tau+t,X(\tau+t,t)G(t,\eta))
\displaystyle \Big|\frac{\partial X(\tau+t,t)G(t,\eta)}{\partial \eta}\Big|\Big\}\Big|_{\tau=0}\\\\
&=&\mathcal{A}(\tau+t,\eta)+\mathcal{B}(\tau+t,\eta)\Big|_{\tau=0},
\end{array}
\end{displaymath}
where $\mathcal{A}(\cdot,\cdot)$ and $\mathcal{B}(\cdot,\cdot)$ are respectively defined by
\begin{displaymath}
\begin{array}{rcl}
\mathcal{A}(\tau+t,\eta)&=&\displaystyle \frac{\partial }{\partial \sigma}\Big\{\rho(\tau+t,X(\tau+t,t)G(t,\eta))\Big\}\Big|\frac{\partial X(\tau+t,t)G(t,\eta)}{\partial \eta}\Big|\\\\
&=&\displaystyle \Big\{\frac{\partial \rho}{\partial \sigma}(\tau+t,X(\tau+t,t)G(t,\eta))+\\\\
& &\triangledown\rho\Big(\tau+t,X(\tau+t,t)G(t,\eta)\Big)A(\tau+t)X(\tau+t,t)G(t,\eta)\Big\}\\\\
& &\displaystyle \Big|\frac{\partial X(\tau+t,t)G(t,\eta)}{\partial \eta}\Big|
\end{array}
\end{displaymath}
and
\begin{displaymath}
\begin{array}{rcl}
\mathcal{B}(\tau+t,\eta)&=&\rho(\tau+t,X(\tau+t,t)G(t,\eta))\displaystyle \frac{\partial}{\partial \sigma}
\Big\{\Big|\frac{\partial X(\tau+t,t)G(t,\eta)}{\partial \eta}\Big|\Big\}\\\\
&=&\rho(\tau+t,X(\tau+t,t)G(t,\eta))\\\\
& &\displaystyle \frac{\partial}{\partial \sigma}\Big\{\Big|\frac{\partial X(\tau+t,t)G(t,\eta)}{\partial G(t,\eta)}\Big|\Big|\frac{\partial G(t,\eta)}{\partial \eta}\Big|\Big\}
\end{array}
\end{displaymath}

As
\begin{displaymath}
\begin{array}{rcl}
\displaystyle \mathcal{A}(t,x)&=&\displaystyle \Big\{\frac{\partial \rho}{\partial t}(t,G(t,\eta))+\triangledown\rho(t,G(t,\eta))A(t)G(t,\eta)\Big\}\Big|\frac{\partial G(t,\eta)}{\partial \eta}\Big|
\end{array}
\end{displaymath}
and
\begin{displaymath}
\begin{array}{rcl}
\mathcal{B}(t,\eta)&=&\displaystyle \rho(t,G(t,\eta))\tr A(t)G(t,\eta)\Big|\frac{\partial G(t,\eta)}{\partial \eta}\Big|,
\end{array}
\end{displaymath}
we can conclude that

\begin{displaymath}
\begin{array}{rcl}
\displaystyle \frac{\partial \bar{\rho}}{\partial t}(t,\eta)+\triangledown \cdot (\bar{\rho}g)(t,\eta)&=&\mathcal{A}(t,\eta)+\mathcal{B}(t,\eta)\\\\
&=&\displaystyle \Big\{\frac{\partial \rho}{\partial t}(t,G(t,\eta))+\triangledown \cdot \rho(t,G(t,\eta)) A(t)G(t,\eta)\Big\}\Big|\frac{\partial G(t,\eta)}{\partial \eta}\Big|,
\end{array}
\end{displaymath}
which is positive since is the product of two positive terms. The positiveness of the first one is ensured by Theorem \ref{densidad-lineal},
while the second follows from the preserving orientation property of $G(t, \cdot)$ as we seen in step 4 of the proof of the Theorem \ref{teorema1}.

\noindent\emph{Step 3:} End of proof.

The existence of a density function for the nonlinear system (\ref{nolin}) is based on the homeomorphisms $H(t,\cdot)$ and $G(t,\cdot)$ constructed in the Theorem \ref{teorema1} and the existence of the density function $\rho(t,x)$ for the linear system (\ref{lin}) from Theorem \ref{densidad-lineal}.
Additionally, in the proof of Theorem \ref{teorema1}, we show that $G$ is a $C^{2}$ preserving orientation
 diffeomorphism while the previous steps we stated that (\ref{densidad-lineal2}) is indeed a density function associated to (\ref{nolin}) and the result follows.
\end{proof}

\begin{remark}
This result follows the lines of \cite[Theorem 3]{CR1} which imposed differentiability of class $C^2$ for $G(t, \cdot)$ as a sufficient condition. Now, as we have seen in Theorem \ref{teorema1}, the restriction to the positive half line allow us to obtain the above conditions in a simpler way in terms of smoothness of  $f(t, \cdot).$
\end{remark}

\end{document}